\documentclass[12pt,a4]{amsart}
\usepackage{amsthm,amsmath,amssymb}
\usepackage{titlesec}
\makeatletter
\def\section{\@startsection{section}{1}%
  \z@{.7\linespacing\@plus\linespacing}{.5\linespacing}%
  {\normalfont\scshape\centering}}
\def\subsection{\@startsection{subsection}{2}%
  \z@{.5\linespacing\@plus.7\linespacing}{-.5em}%
  {\normalfont\bfseries}}
\makeatother
\titleformat*{\section}{\large\bfseries}
\titleformat*{\subsection}{\large\bfseries}
\newtheorem{theorem}{Theorem}[section]

\newtheorem{lemma}{Lemma}[section]

\theoremstyle{remark}
\newtheorem{rem}{Remark}

\title[the Discrete Mean of the Derivative of the Hardy Function]
{On the Discrete Mean of the Derivative of Hardy's $Z$-Function}
\author{Hirotaka Kobayashi}
\date{}
\address{Graduate School of Mathematics, Nagoya University, Furocho, Chikusaku, Nagoya 464-8602, Japan}
\email{m17011z@math.nagoya-u.ac.jp}
\begin{document}

\maketitle

\begin{abstract}
We consider the sum of the square of the derivative of Hardy's $Z$-function over the zeros of Hardy's $Z$-function. If the Riemann Hypothesis is true, it is equal to the sum of $|\zeta'(\rho)|^2$, where $\rho$ runs over the zeros of the Riemann zeta-function.
In 1984, Gonek obtained an asymptotic formula for the sum. In this paper we prove a sharper formula. This result was obtained by Milinovich with a better error term.
\end{abstract}

\section{update: 21/11/2018}
This result was proved by Milinovich \cite{Mil} in his PhD thesis in 2008. Moreover he obtained a better error term. He used $\zeta'(s)$, but we considered $Z'(t)$.

\section{Introduction}

In this paper, we discuss the discrete mean of $Z'(t)$ over the zeros of $Z(t)$,
where $Z(t)$ is Hardy's $Z$-function.
We denote the complex variable by $s=\sigma+it$ with $\sigma, t \in \mathbb{R}$.
Throughout this article, we assume that the Riemann Hypothesis (RH) is true,
and then the mean-value corresponds to that of $\zeta'(\rho)$,
where $\rho$ runs over the zeros of the Riemann zeta-function $\zeta(s)$.

In 1984, Gonek \cite{Go} stated that
\begin{equation}\label{Gonek}
\sum_{0 < \gamma \leq T} \left|\zeta'(\rho) \right|^2
= \frac{1}{24\pi}T \log^{4}\frac{T}{2\pi}+O(T\log^{3}T),
\end{equation}
where $\rho=\frac{1}{2}+i\gamma$ are the zeros of $\zeta(s)$.

On the other hand, Conrey and Ghosh \cite{C&G} showed that 
\begin{equation}
\sum_{0 < \gamma \leq T} \max_{\gamma < t \leq \gamma^{+}} \left|\zeta \left(\frac{1}{2}+it \right) \right|^2
\sim \frac{e^2-5}{4\pi}T\log^{2}\frac{T}{2\pi},
\end{equation}
where $\gamma$ and $\gamma^{+}$ are ordinates of consecutive zeros of $\zeta(s)$.
This summand means the extremal value of Hardy's $Z$-function.
They calculated the integral
\begin{equation*}
\frac{1}{2\pi i}\int_{C} \frac{Z_1'}{Z_1}(s)\zeta(s)\zeta(1-s)ds,
\end{equation*}
where $C$ is positively oriented rectangular path with vertices $c+i$, $c+iT$, $1-c+iT$ and $1-c+i$ where $c=\frac{5}{8}$, and $Z_1(s)$ is defined by
\begin{equation}
Z_1(s) :=\zeta'(s)-\frac{1}{2}\omega(s)\zeta(s)
\end{equation}
with
\begin{equation*}
\omega(s)=\frac{\chi'}{\chi}(s),
\end{equation*}
where $\chi(s)=2^s \pi^{s-1}\sin \left(\frac{\pi s}{2} \right)\Gamma(1-s)$.

\begin{rem}
Actually, they considered the rectangle whose imaginary part is $T\leq \Im{s}\leq T+T^{3/4}$ for convenience.
\end{rem}

In this article, by considering the integral
\begin{equation*}
\frac{1}{2\pi i}\int_{C} \frac{\zeta'}{\zeta}(s)Z_1(s)Z_1(1-s)ds,
\end{equation*}
we will prove the following theorem.
\begin{theorem}
If the RH is true, then for any sufficiently large $T$,
\begin{equation}
\begin{split}
\sum_{0 <\gamma \leq T} \left|Z'(\gamma) \right|^2
&= \frac{1}{24\pi} T\log^{4}\frac{T}{2\pi}+\frac{2\gamma_0-1}{6\pi}T\log^{3}\frac{T}{2\pi} \\
&\quad +a_1T\log^{2}\frac{T}{2\pi}+a_2T\log\frac{T}{2\pi}+a_3T + O(T^{\frac{3}{4}}\log^{\frac{7}{2}}T),
\end{split}
\end{equation}
where $\gamma_0$ is the Euler constant and $a_i \ (i=1, 2, 3)$ are constants that can be explicitly expressed by the Stieltjes constants and the summation is over the zeros $\gamma$ of $Z(t)$ with the multiplicity.
\end{theorem}

Actually, the sum on the left-hand side is equal to the sum on the left-hand side of $(\ref{Gonek})$. Hence this theorem is an improvement on Gonek's result.

In the proof, we apply Hall's result \cite{Hall} that for each $k=0, 1, 2, \cdots$, and any sufficiently large $T$,
\begin{equation}
\int_{0}^{T} Z^{(k)}(t)^2dt
=\frac{1}{4^k(2k+1)}TP_{2k+1}\left(\log \frac{T}{2\pi}\right)+O(T^{\frac{3}{4}}\log^{2k+\frac{1}{2}}T),
\end{equation}
where $P_{2k+1}(x)$ is the monic polynomial of degree $2k+1$ given by
\begin{equation*}
P_{2k+1}(x)=W_{2k+1}(x)+(4k+2)\sum_{h=0}^{2k} \binom {2k}{h}(-2)^h \gamma_h W_{2k-h}(x),
\end{equation*}
in which
\begin{equation*}
W_{g}(v)=\frac{1}{e^v}\int_{0}^{e^v}\log^g u du,
\quad \zeta(s)=\frac{1}{s-1}+ \sum_{h=0}^{\infty}\frac{(-1)^h \gamma_h}{h!}(s-1)^h.
\end{equation*}
The $\gamma_h$ are called the Stieltjes constants.

\section{Preliminary lemmas}

We prepare some lemmas for our proof.
\begin{lemma}
$Z_1(s)$ has the following properties.
\begin{enumerate}
\item If $s=\frac{1}{2}+it$, then $|Z_1(s)|=|Z'(t)|$.
\vspace{2.5mm}
\item $Z_1(s)$ satisfies the functional equation $-Z_1(s)=\chi(s)Z_1(1-s)$ for all $s$.
\end{enumerate}
\end{lemma}

\begin{proof}
See the proof of the lemma in \cite{C&G}.
\end{proof}

By Stirling's formula, we can show that
\begin{lemma}
For $|\arg s|<\pi$ and $t\geq1$, we have
\begin{equation}
\chi(1-s)=e^{-\frac{\pi i}{4}}\left(\frac{t}{2\pi} \right)^{\sigma-\frac{1}{2}}\exp \left(it\log \frac{t}{2\pi e}\right) \left(1+O\left(\frac{1}{t}\right)\right)
\end{equation}
and
\begin{equation}\label{log-stirling}
\frac{\chi'}{\chi}(s)=-\log \frac{t}{2\pi}+O\left(\frac{1}{t} \right).
\end{equation}
\end{lemma}

If RH is true, then the Lindel\"{o}f Hypohesis is also true.
Therefore we can obtain the following estimates. 
\begin{lemma}\label{lindelof}
If the RH is true, then for $|t|\geq1$,
\begin{equation*}
\zeta(s) \ll \begin{cases} 1 & \text{$1< \sigma$,} \\
|t|^{\varepsilon} & \text{$\frac{1}{2}\leq \sigma \leq1$,} \\
|t|^{\frac{1}{2}-\sigma+\varepsilon} & \text{$\sigma <\frac{1}{2}$.}
\end{cases}
\end{equation*}
\end{lemma}

As in the paper of Conrey and Ghosh \cite{C&G}, we apply the following lemma by Gonek \cite{Go},
\begin{lemma}[Gonek]\label{G'slem}
Let $\{b_n \}_{n=1}^{\infty}$ be a sequence of complex numbers such that $b_n \ll n^{\varepsilon}$ for any $\varepsilon >0$. Let $a>1$ and let $m$ be a non-negative integer. Then for any sufficiently large $T$,
\begin{equation*}
\begin{split}
&\quad \frac{1}{2\pi}\int_{1}^{T}\left(\sum_{n=1}^{\infty}b_n n^{-a-it} \right)\chi(1-a-it)
\left(\log \frac{t}{2\pi} \right)^m dt \\
&= \sum_{1\leq n\leq T/2\pi} b_n(\log n)^m+O(T^{a-\frac{1}{2}}(\log T)^m).
\end{split}
\end{equation*}
\end{lemma}

\section{The proof of the theorem}

It is sufficient to prove the theorem in the case where $T$ does not coincide with the ordinate of the zeros of $\zeta(s)$. By the assumption of RH, we can consider $m|Z'(\gamma)|^2$ as the residue of $\frac{\zeta'}{\zeta}(s)Z_1(s)Z_1(1-s)$ at the zero $\frac{1}{2}+i\gamma$ of the Riemann $\zeta$-function, where $m$ is the multiplicity of the zero.
Hence when we denote our sum by $M(T)$, by the residue theorem, we see that
\begin{equation*}
M(T)=\frac{1}{2\pi i}\int_{C}\frac{\zeta'}{\zeta}(s)Z_1(s)Z_1(1-s)ds.
\end{equation*}
By Lemma \ref{lindelof}, the integral on the horizontal line can be estimated as $O(T^{\frac{1}{8}+\varepsilon})$. Therefore it satisfies that
\begin{align*}
M(T)
&= \frac{1}{2\pi i}\int_{c+i}^{c+iT} \frac{\zeta'}{\zeta}(s)Z_1(s)Z_1(1-s)ds \\
&\quad + \frac{1}{2\pi i}\int_{1-c+iT}^{1-c+i}\frac{\zeta'}{\zeta}(s)Z_1(s)Z_1(1-s)ds
+O(T^{\frac{1}{8}+\varepsilon}) \\
       &= I_{1}+I_{2}+O(T^{\frac{1}{8}+\varepsilon}),
\end{align*}
say.
On the integral $I_2$,
\begin{align*}
I_2&= -\frac{1}{2\pi i}\int_{1-c+i}^{1-c+iT}\frac{\zeta'}{\zeta}(s)Z_1(s)Z_1(1-s)ds \\
    &= -\frac{1}{2\pi i}\int_{1-c+i}^{1-c+iT}\left(\frac{\chi'}{\chi}(s)-\frac{\zeta'}{\zeta}(1-s)\right)Z_1(s)Z_1(1-s)ds \\
    &= -\frac{1}{2\pi i}\int_{1-c+i}^{1-c+iT}\frac{\chi'}{\chi}(s)Z_1(s)Z_1(1-s)ds \\
    &\quad +\frac{1}{2\pi i}\int_{1-c+i}^{1-c+iT}\frac{\zeta'}{\zeta}(1-s)Z_1(s)Z_1(1-s)ds.
\end{align*}
When we replace $s$ by $1-s$, the second integral is 
\begin{equation*}
-\frac{1}{2\pi i}\int_{c-i}^{c-iT}\frac{\zeta'}{\zeta}(s)Z_1(s)Z_1(1-s)ds= \overline{I_1}.
\end{equation*}
Now we see that 
\[M(T)=-\frac{1}{2\pi i}\int_{1-c+i}^{1-c+iT}\frac{\chi'}{\chi}(s)Z_1(s)Z_1(1-s)ds+2\Re{I_1}+O(T^{\frac{1}{8}+\varepsilon}). \]
By Cauchy's integral theorem, the first integral is equal to
\[-\frac{1}{2\pi i}\int_{\frac{1}{2}+i}^{\frac{1}{2}+iT}\frac{\chi'}{\chi}(s)Z_1(s)Z_1(1-s)ds+O(T^{\frac{1}{8}+\varepsilon}). \]
This error term is derived from the integral on the horizontal line.
By (\ref{log-stirling}) and Lemma \ref{lindelof} we see that the above integral is
\begin{equation}\label{mean-value}
\frac{1}{2\pi} \int_{1}^{T}\log \frac{t}{2\pi} Z'(t)^2dt +O(T^{\varepsilon}).
\end{equation}
Therefore when we put
\begin{equation*}
I(t)=\int_{1}^{t}Z'(x)^2dx,
\end{equation*}
using integration by parts and Hall's result, we can show that the integral in (\ref{mean-value}) is equal to
\begin{equation*}
\begin{split}
&\quad \frac{1}{2\pi} \log \frac{T}{2\pi}I(T)-\frac{1}{2\pi}\int_{1}^{T}t^{-1}I(t)dt \\
&=\frac{1}{24\pi}T\log \frac{T}{2\pi}P_3\left(\log \frac{T}{2\pi} \right) \\
&\quad -\frac{1}{24\pi}\int_{1}^{T}P_3\left(\log \frac{t}{2\pi} \right)dt+O(T^{\frac{3}{4}}\log^{\frac{7}{2}}T),
\end{split}
\end{equation*}
and $P_3(x)$ is explicitly written as
\begin{equation}\label{P}
\begin{split}
P_3(x)&=x^3+3(2\gamma_0-1)x^2 \\
&\quad-6(2\gamma_0+4\gamma_1-1)x+6(2\gamma_0+4\gamma_1+4\gamma_2-1).
\end{split}
\end{equation}
Then we find
\begin{equation}
\begin{split}
&\quad -\frac{1}{2\pi i}\int_{1-c+i}^{1-c+iT}\frac{\chi'}{\chi}(s)Z_1(s)Z_1(1-s)ds \\
&=\frac{1}{24\pi}T\log^4\frac{T}{2\pi}+\frac{3\gamma_0-2}{12\pi}T\log^3\frac{T}{2\pi} \\
&\quad +b_1T\log^2\frac{T}{2\pi}+b_2T\log\frac{T}{2\pi}+b_3T+O(T^{\frac{3}{4}}\log^{\frac{7}{2}}T),
\end{split}
\end{equation}
where $b_i \, (i=1,2,3)$ are constants derived from the coefficients of (\ref{P}).

We calculate $I_1$ to complete the proof. By Cauchy's integral theorem,
\begin{align*}
I_1 &= -\frac{1}{2\pi i}\int_{b+i}^{b+iT}\frac{\zeta'}{\zeta}(s)\chi(1-s)Z_1(s)^2ds+O(T^{b-\frac{1}{2}+\varepsilon}) \\
    &= -\frac{1}{2\pi i}\int_{b+i}^{b+iT}\frac{\zeta'}{\zeta}(s)\chi(1-s)\zeta'(s)^2ds \\
    &\quad +\frac{1}{2\pi i}\int_{b+i}^{b+iT}\omega(s)\zeta'(s)^2\chi(1-s)ds \\
    &\quad -\frac{1}{8\pi i}\int_{b+i}^{b+iT}\omega(s)^2\zeta(s)\zeta'(s)\chi(1-s)ds+O(T^{b-\frac{1}{2}+\varepsilon}) \\
    &= I_{I}+I_{II}+I_{III}+O(T^{b-\frac{1}{2}+\varepsilon}),
\end{align*}
say, where $b=\frac{9}{8}$ and the error term is derived from the integral on the horizontal line.

By applying Lemma \ref{G'slem},
\begin{align*}
I_{I} &= \frac{1}{2\pi i}\int_{b+i}^{b+iT} \left(\sum_{m=1}^{\infty} \frac{\Lambda(m)}{m^s} \right)\left(\sum_{n=1}^{\infty}\frac{D(n)}{n^s} \right)\chi(1-s)ds \\
     &= \frac{1}{2\pi}\int_{1}^{T} \left(\sum_{m=1}^{\infty} \frac{\Lambda(m)}{m^{b+it}} \right)\left(\sum_{n=1}^{\infty}\frac{D(n)}{n^{b+it}} \right)\chi(1-b-it)dt \\
     &= \sum_{1 \leq mn \leq \frac{T}{2\pi}} \Lambda(m)D(n) +O(T^{b-\frac{1}{2}}),
\end{align*}
where
\[D(n)=\sum_{d\mid n}\log d\log\frac{n}{d}. \]
Using Perron's formula and the residue theorem,
\begin{align*}
\sum_{mn \leq x} \Lambda(m)D(n)
&= -\frac{1}{2\pi i}\int_{b-iT}^{b+iT} \frac{\zeta'}{\zeta}(s)\zeta'(s)^2\frac{x^s}{s}ds
+O(x^{\varepsilon})+R \\
&=-\underset{s=1}{\rm Res}\frac{\zeta'}{\zeta}(s)\zeta'(s)^2\frac{x^s}{s}
-\frac{1}{2\pi i}\int_{c-iT}^{c+iT}\frac{\zeta'}{\zeta}(s)\zeta'(s)^2\frac{x^s}{s}ds \\
&\quad+\frac{1}{2\pi i}\int_{b+iT}^{c+iT}\frac{\zeta'}{\zeta}(s)\zeta'(s)^2\frac{x^s}{s}ds \\
&\quad+\frac{1}{2\pi i}\int_{c-iT}^{b-iT}\frac{\zeta'}{\zeta}(s)\zeta'(s)^2\frac{x^s}{s}ds+O(x^{\varepsilon})+R \\
&= -\underset{s=1}{\rm Res}\frac{\zeta'}{\zeta}(s)\zeta'(s)^2\frac{x^s}{s}+O(x^cT^{\varepsilon}+x^bT^{-1+\varepsilon})+R,
\end{align*}
where $R$ is the error term appearing in Perron's formula (see \cite[p.139]{Mon&Vau}) and satisfies that
\begin{align*}
R &= \frac{1}{\pi}\sum_{x/2<mn<x}\Lambda(m)D(n)\ {\rm si} \left(T\log \frac{x}{mn}\right) \\
   &\quad -\frac{1}{\pi}\sum_{x<mn<2x}\Lambda(m)D(n)\ {\rm si} \left(T\log \frac{x}{mn}\right)
+O\left(\frac{(4x)^b}{T}\sum_{mn=1}^{\infty}\frac{|\Lambda(m)D(n)|}{(mn)^b} \right) \\
   &\ll \sum_{\substack{x/2<mn<2x \\ mn \neq x}} |\Lambda(m)D(n)| \min \left(1, \frac{x}{T|x-mn|}\right)+\frac{(4x)^b}{T}\sum_{mn=1}^{\infty}\frac{|\Lambda(m)D(n)|}{(mn)^b}
\end{align*}
with
\begin{equation*}
{\rm si}(x)=-\int_{x}^{\infty}\frac{\sin u}{u}du.
\end{equation*}
By introducing the Dirichlet convolution and considering $x$ as a half integer, the error term $R$ can be estimated as follows;
\begin{align*}
R&\ll\frac{x}{T}\sum_{x/2<n<2x}\left|\frac{{\bf \Lambda} * {\bf D}(n)}{x-n} \right|
+\frac{(4x)^b}{T}\sum_{n=1}^{\infty}\frac{\left|{\bf \Lambda} * {\bf D}(n)\right|}{n^b} \\
&\ll \frac{x^{\varepsilon}}{T}\sum_{x/2<n<2x}\left|\frac{1}{1-\frac{n}{x}} \right|+\frac{x^b}{T} \\
&\ll \frac{x^{\varepsilon}}{T}\log x +\frac{x^b}{T}.
\end{align*}
Therefore we obtain
\begin{equation*}
\sum_{mn \leq x} \Lambda(m)D(n)=-\underset{s=1}{\rm Res}\frac{\zeta'}{\zeta}(s)\zeta'(s)^2\frac{x^s}{s}+O(x^bT^{-1+\varepsilon}+x^cT^{\varepsilon}).
\end{equation*}
This residue is
\begin{align*}
&-\frac{1}{4!}x(\log^4x-4\log^3x+12\log^2x-24\log x+24) \\
&+\frac{\eta_0}{3!}x(\log^3x-3\log^2x+6\log x-6) \\
&+\left(\gamma_1+\frac{\eta_1}{2}\right)x(\log^2x-2\log x+2) \\
&+(\eta_2+4\gamma_2-2\gamma_1\eta_0)x(\log x-1) \\
&+(\eta_3+6\gamma_3-\gamma_1^2-2\eta_1\gamma_1)x,
\end{align*}
where $\eta_k$ are defined by 
\begin{equation*}
\frac{\zeta'}{\zeta}(s)=-\frac{1}{s-1}+\sum_{k=0}^{\infty}\eta_k(s-1)^k
\end{equation*}
and can be expressed by the Stieltjes constants.
The other integrals can be calculated in the same way.  Taking $x=\frac{T}{2\pi }$, we obtain
\[2\Re I_1= \frac{\gamma_0}{12\pi}T\log^{3}\frac{T}{2\pi}+c_1T\log^{2}\frac{T}{2\pi}
+c_2T\log\frac{T}{2\pi}+c_3T+O(T^{c+\varepsilon}), \]
where $c_i \ (i=1,2,3)$ can be expressed by the Stieltjes constants.
Hence 
\begin{equation*}
\begin{split}
M(T)
&= \frac{1}{24\pi} T\log^{4}\frac{T}{2\pi}+\frac{2\gamma_0-1}{6\pi}T\log^{3}\frac{T}{2\pi} \\
&+a_1T\log^{2}\frac{T}{2\pi}+a_2T\log\frac{T}{2\pi}+a_3T + O(T^{\frac{3}{4}}\log^{\frac{7}{2}}T),
\end{split}
\end{equation*}
where $a_i \ (i=1,2,3)$ is the sum of $b_i$ and $c_i$,
and the proof is completed.

\section*{Acknowledgement}
I would like to thank my supervisor Professor Kohji Matsumoto for useful advice.
I am grateful to the seminar members  for indications by which my argument can be sophisticated. I would also like to thank Micah Baruch Milinovich of making me aware of his result.

\end{document}